\documentclass[10pt,runningheads,a4paper]{llncs}
\usepackage{graphicx}
\usepackage{amsmath}
\usepackage{amssymb}
\usepackage{color}
\numberwithin{theorem}{subsection}
\numberwithin{corollary}{section}
\numberwithin{definition}{subsection}
\begin{document}
\title{On Topological Properties of Planar Octahedron Networks}
\author{Haidar Ali, Farzana Kousar}

\institute{Department of Mathematics,
Government College University,\\
Faisalabad, Pakistan\\
E-mail: \{haidarali@gcuf.edu.pk\}\\
$^*$Correspondence: haidarali@gcuf.edu.pk}
\authorrunning{Haidar Ali, Farzana Kousar}
\titlerunning{On Topological Properties of Planar Octahedron Networks} \maketitle

\markboth{\small{Haidar Ali}} {\small{On Topological Properties of Planar Octahedron Networks}}

\begin{abstract}
Topological indices are scientific details of graphs which represents its topology and of the most part graph invariant.  In QSAR/QSPR, physico-chemical characteristics and topological indices, for example, atom bond connectivity $(ABC)$ and geometric-arithmetic  $(GA)$ indices are apply to foresee the bioactivity of concoction mixes. Graph theory discovered a significant practice in the region of investigation. In this paper, we are taking Planar Octahedron networks, produced by honeycomb structure of dimension $n$ and obtain analytical closed results of Multiplicative topological indices for  firstly and presents closed formulas of  degree based indices.

\end{abstract}
\par
\centerline{{\bf  Mathematics Subject Classification: 05C12, 05C90}}
\vspace{.3cm} {\bf Keywords:}  Topological indices, Planar Octahedron network, Triangular Prism network, Dominating Planar Octahedron network,$R_{\alpha}$, $M _{1}$, $ABC$, $GA$, $ABC_{4}$, $GA_{5}$.

\section[Introduction and preliminary results]{Introduction and preliminary results}
Topological indices are valuable instruments that are given by graph theory. Particles and atomic mixes are frequently demonstrated by sub-atomic graph. In perspective on graph hypothesis, a sub-atomic chart depicts the auxiliary recipe of synthetic compound by changing over molecules by vertices and concoction bonds by edges.  The blend of chemical science, arithmetic and data science is a cheminformatics. It predicts the property of chemical compound by using Quantitative structure-activity (QSAR)and structure property (QSPR) relationships. The topological indices are used to predict chemical compound bioactivity.

A representation of numbers, polynomials and matrices are representations of a graph. Graph has its own characteristics which can be determined by topological indices and the topology of graph remains unchanged under automorphism of graph. In the different classes of indices, degree based topological indices are of extraordinary significance and assume an essential job in substance chart hypothesis and especially in science. In increasingly exact manner, a topological index $Top(H)$ of a graph, is a number with the property that for each graph $G$ isomorphic to $H$, $Top(H) = Top(G)$. The idea of topological index originated from Wiener {\cite{Wiener}} while he was dealing with boiling point of paraffin, named this record as \emph{path number}. Later on, renamed as \emph{Wiener index} \cite{Deza}.\\\\
\textbf{Drawing Algorithm for Planar Octahedron (POH) Networks}\\
Step 1: Construct a Silicate network of dimension n.\\
Step 2: Fix new vertices at the centroid of each triangular face and connect it
to oxide vertices in the corresponding triangle face.\\
Step 3: Connect all these new centroid vertices which lie in the same silicate
cell.\\
Step 4: Remove all Silicon vertices. The resulting graph is called planar
Octahedron network of dimension n.\\
Also we can derive Triangular Prism network $TP$ and Dominating Planar Octahedron network $DPOH$.\\

\begin{figure}[h]
\centering
\includegraphics[width=7cm]{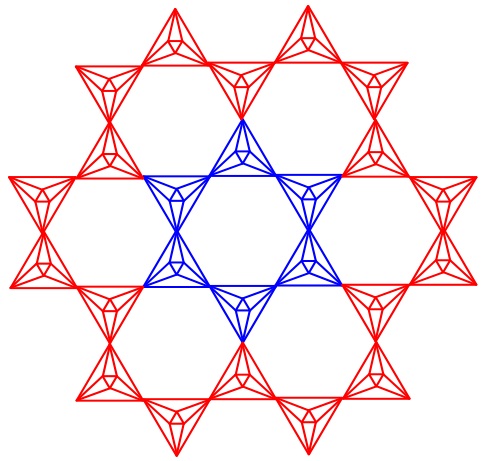}\\
\caption{Planar Octahedron network $POH(2)$.}\label{fig 1}
\end{figure}
\unskip
\begin{figure}[h]
\centering
\includegraphics[width=10cm]{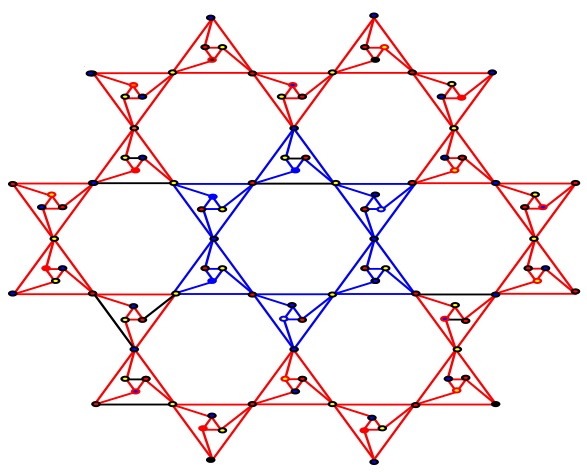}\\
\caption{Triangular Prism network $TP(2)$.}\label{fig 2}
\end{figure}
\unskip
\begin{figure}[h]
\centering
\includegraphics[width=10cm]{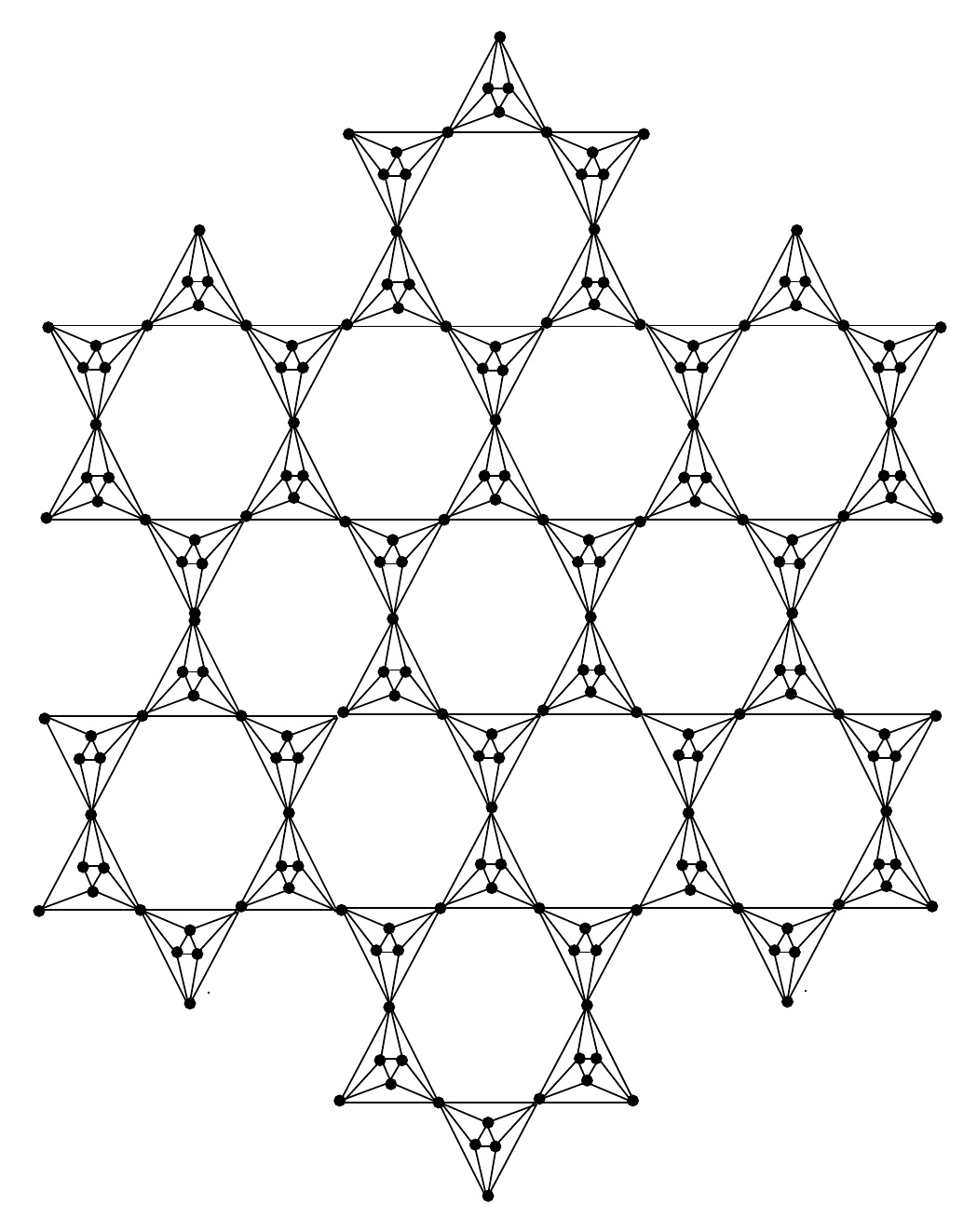}\\
\caption{Dominating Planar Octahedron network $DPOH(2)$.}\label{fig 3}
\end{figure}
In this article, $G$ is considered to be a network with vertex set $V(G)$ and edge set $E(G)$, the $deg(u)$ is the degree of vertex $u\in V(G)$ and
$S_{u}={\sum\limits_{v\in N_{G}(u)}deg(v)}$ where $N_{G}(u)=\{v\in V(G)\mid uv \in E(G)\}$. The notations used in this article are mainly taken from
the books \cite{Diudea,Gutman}.

Let $G$ be a connected graph. Then the Wiener index of $G$ is defined as
 \begin{equation}
 W(G)=\frac{1}{2}\sum\limits_{(u,v)}d(u,v),
\end{equation}
where $(u,v)$ is any ordered pair of vertices in $G$ and $d(u,v)$ is $u-v$ geodesic.

The very first and oldest degree based topological index is \emph{Randi\'{c}} index \cite{Randic} denoted by $R_{-\frac{1}{2}}(G)$ and was introduced
by Milan Randi\'{c} and is defined as
\begin{equation}
 R_{-\frac{1}{2}}(G)=\sum\limits_{uv \in E(G)}\frac{1}{\sqrt{deg(u)deg(v)}}.
\end{equation}
The general Randi\'{c} index  $R_{\alpha}(G)$ is the sum of $(deg(u)deg(v))^{\alpha}$ over all the edges $e=uv\in E(G) $ and is defined as
\begin{equation}
 R_{\alpha}(G)=\sum\limits_{uv \in E(G)}(deg(u)deg(v))^{\alpha}\,\,\ \text{for}\,\,\ \alpha = 1, \frac{1}{2}, -1, -\frac{1}{2}.
\end{equation}
An important topological index introduced by Ivan Gutman and Trinajsti\'{c} is the \emph{Zagreb} index denoted by $M_{1}(G)$ and is defined as
\begin{equation}
 M_{1}(G)=\sum_{uv \in E(G)}(deg(u)+deg(v)).
\end{equation}
One of the well-known degree based topological index is \emph{atom-bond connectivity} $(ABC)$ index introduced by Estrada \emph{et al.} in \cite{Estrada}
and defined as
\begin{equation}
ABC(G)=\sum_{uv \in E(G)}\sqrt{\frac{deg(u)+deg(v)-2}{deg(u)deg(v)}}.
\end{equation}
Another well-known connectivity topological descriptor is \emph{geometric-arithmetic} $(GA)$ index which was introduced
by Vuki\v{c}evi\'{c} \emph{et al.} in \cite{Vukičević} and defined as
\begin{equation}
 GA(G)= \sum_{uv \in E(G)} \frac{2\sqrt{deg(u)deg(v)}}{(deg(u)+deg(v))}.
\end{equation}
The fourth version of atom-bond connectivity index denoted by $ABC_{4}$ and fifth version of geometric-arithmetic index denoted by $GA_{5}$ can be computed if we are able to find the edge partition of these interconnection networks based on sum of the degrees of end vertices of each edge in these graphs. The fourth version of $ABC$ index is introduced by Ghorbani \emph{et al.} \cite{Graovac1} and defined as
\begin{equation}
ABC_{4}(G)=\sum_{uv \in E(G)} \sqrt{\frac{S_{u}+S_{v}-2}{S_{u}S_{v}}}.
\end{equation}
Recently, the fifth version of $GA$ index is proposed by Graovac \emph{et al.} \cite{Graovac2} and is defined as
\begin{equation}
GA_{5}(G)=\sum_{uv \in E(G)}\frac{2\sqrt{S_{u}S_{v}}}{(S_{u}+S_{v})}.
\end{equation}

\noindent The general Randi\'{c} index for $\alpha=1$ is the second Zagreb index for any graph $G$.

\section{Main results}

In this article, we study the general Randi\'{c}, first Zagreb, $ABC$, $GA$, $ABC_{4}$ and $GA_{5}$ indices and give exact results of these indices for Planar Octahedron $POH(n)$ network, Triangular Prism $ TP(n)$ network and Dominating Planar Octahedron $DPOH(n)$ network. Nowadays there is an extensive research activity on $ABC$ and $GA$ indices and their variants, for further study of topological indices of various graphs, \cite{Ali1,Ali2,Akhter1,Baig,Gao,Imran2,Simonraj,Shirdel,Wei}.

\subsection{Results for Planar Octahedron $ POH(n) $ networks}
In this section, we calculate certain degree based topological indices of Planar Octahedron  $POH(n)$ network. We compute general Randi\'{c} $R_{\alpha}(G)$ with $\alpha=\{1,-1,\frac{1}{2},-\frac{1}{2}\}$, $ABC$, $GA$, $ABC_{4}$ and $GA_{5}$ indices for $POH(n)$ network.

\begin{theorem}	
Consider the Planar Octahedron network $G\cong POH(n)$ for $n\in\mathbb{N}$. Then its general Randi\'{c} index is equal to
	\begin{equation*}
	R_{\alpha}(POH(n))=\left\{
	\begin{array}{ll}
	288n(-2+9n), & \hbox{$\alpha=1$;} \\
	24n(-2+(9+6\sqrt{2})n), & \hbox{$\alpha=\frac{1}{2}$;} \\
	\frac{9}{32}n(2+9n), & \hbox{$\alpha=-1$;}\\
	\frac{3}{4}n(2+(9+6\sqrt{2})n), & \hbox{$\alpha=-\frac{1}{2}$.}
	\end{array}
	\right.
	\end{equation*}
\end{theorem}
\begin{proof}
	Let $G$ be the Planar Octahedron network. The Planar Octahedron  network $POH(n)$ has $27n^2-21n$ vertices and the edge set of $POH(n)$ is divided into three partitions based on the degree of end vertices. Table 1 shows such an edge partition of $POH(n)$. Thus from $(3)$ is follows that,
	\begin{equation*}
	R_{\alpha}(G)=\sum\limits_{uv \in E(G)}(deg(u)deg(v))^{\alpha}.
	\end{equation*}
	\noindent \textbf {For $\alpha=1$}\\
	Now we apply the formula,
	\begin{equation*}
	R_{1}(G)=\sum_{uv \in E_j(G)} \sum_{j=1}^3deg{(u)} \cdot deg{(v)}.
	\end{equation*}
	By using edge partition given in Table 1, we get
	$$R_{1}(G)=16|E_1(POH(n))|+32|E_2(POH(n))|+64|E_3(POH(n))|,$$
	$$\Longrightarrow R_{1}(G)= 288n(-2+9n).$$
	\\
	\textbf{For $\alpha=\frac{1}{2}$}\\
	We apply the formula,
	\begin{equation*}
	R_{\frac{1}{2}}(G)=\sum_{uv \in E_j(G)}\sum_{j=1}^3\sqrt{deg(u) \cdot deg(v)}.
	\end{equation*}
	By using edge partition given in Table 1, we get
	$$R_{\frac{1}{2}}(G)=4|E_1(POH(n))|+4\sqrt{2}|E_2(POH(n))|+8|E_3(POH(n))|,$$
	$$\Longrightarrow R_{\frac{1}{2}}(G)= 24n(-2+(9+6\sqrt{2})n).$$
	\\
	\textbf{For $\alpha=-1$}\\
	We apply the formula,
	\begin{equation*}
	R_{-1}(G)=\sum_{uv \in E_j(G)}\sum_{j=1}^3\frac{1}{deg(u) \cdot deg(v)}.
	\end{equation*}
	
	$$R_{-1}(G)=\frac{1}{16}|E_1(POH(n))|+\frac{1}{32}|E_2(POH(n))|+\frac{1}{64}|E_3(POH(n))|,$$\\
By using edge partition given in Table 1, we get\\
	$$\Longrightarrow R_{-1}(G)=  \frac{9}{32}n(2+9n).$$
	\\
	\textbf{For $\alpha=-\frac{1}{2}$}\\
	We apply the formula,
	\begin{equation*}
	R_{-\frac{1}{2}}(G)=\sum_{uv \in E_j(G)}\sum_{j=1}^3\frac{1}{\sqrt{deg(u). \cdot deg(v)}}.
	\end{equation*}
	$$R_{-\frac{1}{2}}(G)=\frac{1}{4}|E_1(POH(n))|+\frac{1}4{\sqrt{2}}|E_2(POH(n))|+\frac{1}{8}|E_3(POH(n))|,$$\\
By using edge partition given in Table 1, we get\\
	\begin{equation*}
	\Longrightarrow R_{-\frac{1}{2}}(G)= \frac{3}{4}n(2+(9+6\sqrt{2})n).
	\end{equation*}
\end{proof}

In the following theorem, we compute first Zagreb index of Planar Octahedron  network $POH(n)$.
\begin{theorem}
	For Planar Octahedron network $G\cong POH(n)$ for $n\in\mathbb{N}$. Then its first Zagreb index is equal to
	\begin{equation*}
	M_{1}(POH(n))=96n(-1+9n).
	\end{equation*}
\end{theorem}
\begin{proof}
	Let $G$ be the Planar Octahedron network $POH(n)$. By using edge partition from Table 1, the result follows.
	\begin{equation*}
	M_{1}(POH(n))=\sum_{uv \in E(G)}(deg(u)+deg(v))=\sum_{uv \in E_j(G)}\sum_{j=1}^3(deg(u)+deg(v)).
	\end{equation*}
	$$M_{1}(POH(n))=8|E_1(POH(n))|+12|E_2(POH(n))|+16|E_3(POH(n))|,$$
	By doing some calculation, we get\\
	$$\Longrightarrow M_{1}(POH(n))=96n(-1+9n).$$
\end{proof}

Now, we compute $ABC$, $GA$, $ABC_{4}$ and $GA_{5}$ indices of Planar Octahedron  network $POH(n)$.
\begin{theorem}
	Let $D\cong POH(n)$ be the Planar Octahedron network, then\\ \\
	$\bullet$ $ABC(G)$ = $\frac{3}{4}n(4\sqrt{6}-2\sqrt{14}+3(4\sqrt{5}+2\sqrt{6}+\sqrt{14})n)$, for $n\in\mathbb{N}$ \\ \\
	$\bullet$ $GA(G)$ = $12(3+2\sqrt{2})n^{2}$, for $n\in\mathbb{N}$\\ \\
	$\bullet$ $ABC_4(G)$ = $\frac{1}{440}(264\sqrt{29}+88\sqrt{930}+24\sqrt{2255}+5280(-1+n)+
	60\sqrt{330}(-1+n)n+24\sqrt{3410}(-1+n)+165\sqrt{94}(-1+n)^{2}+528\sqrt{35}n+132\sqrt{38}n+330\sqrt{46}(-1+n)n+60\sqrt{86}(-3+2n)+220\sqrt{35}(2-5n+3n^{2}))$, for $n\geq2$\\ \\
	$\bullet$ $GA_5$=$8\sqrt{2}+6\sqrt{15}+\frac{8}{7}\sqrt{110}+\frac{48}{23}\sqrt{33}(-1+n)+\frac{3}{2}\sqrt{55}(-1+n)+\frac{96}{17}\sqrt{66}(-1+n)-36n+\frac{48}{11}\sqrt{30}n+36n^{2}+8\sqrt{2}(2-5n+3n^{2})$, for $n\geq2$
\end{theorem}
\begin{proof}
	By using edge partition given in Table 1, we get the result. From $(5)$ it follows that,
	$$ABC(G)=\sum_{uv \in E(G)}\sqrt{\frac{deg(u)+deg(v)-2}{deg(u)\cdot deg(v)}}=\sum_{uv \in E_j(G)}\sum_{j=1}^3\sqrt{\frac{deg(u)+deg(v)-2}{deg(u)\cdot deg(v)}}.$$
	$$ABC(G)=\frac{\sqrt{6}}{4}|E_1(POH(n))|+\dfrac{\sqrt{5}}{4}|E_2(POH(n))|+\dfrac{\sqrt{14}}{8}|E_3(POH(n))|,$$
	By doing some calculation, we get\\
	$$\Longrightarrow ABC(G)=\frac{3}{4}n(4\sqrt{6}-2\sqrt{14}+3(4\sqrt{5}+2\sqrt{6}+\sqrt{14})n).$$
	from (6) we get
	$$GA(G)= \sum_{uv \in E(G)} \frac{2\sqrt{deg(u)deg(v)}}{(deg(u)+deg(v))}=\sum_{j=1}^3\sum_{uv \in E_j(G)}\frac{2\sqrt{deg(u)deg(v)}}{(deg(u)+deg(v))}.$$
	By doing some calculation, we get\\
	$$GA(G)=|E_1(POH(n))|+\frac{2\sqrt{2}}{3}|E_2(POH(n))|+|E_3(POH(n))|,$$
	$$\Longrightarrow GA = \frac{4}{3}n((9+6\sqrt{2})n+2\sqrt{2}-3)$$
	Let us consider an edge partition based on degree sum of neighbors of end vertices. Then the edge set $E(POH(n))$ can be divided into twelve edge partitions $E_j(POH(n)),4\leq j \leq15$. Table 2 shows such edge partitions.\\
	From (7) we get\\
	$$ABC_{4}(G)=\sum_{uv \in E(G)} \sqrt{\frac{S_{u}+S_{v}-2}{S_{u}S_{v}}}=\sum_{j=4}^{15}\sum_{uv \in E_j(G)}\sqrt{\frac{S_{u}+S_{v}-2}{S_{u}S_{v}}}.$$\\
	\begin{eqnarray*}
ABC_{4}(G)&=&\frac{\sqrt{38}}{20}|E_4(POH(n))|+\frac{\sqrt{35}}{20}|E_5(POH(n))|+\frac{\sqrt{29}}{20}|E_6(POH(n))|+\\&&0.2654|E_7(POH(n))|+
	\frac{\sqrt{46}}{4}|E_8(POH(n))|+\\&&0.2541|E_9(POH(n))|+\frac{1}{4}|E_{10}(POH(n))|+\\&&\frac{\sqrt{35}}{24}|E_{11}(POH(n))|+
	0.2518|E_{12}(POH(n))|+\\&&\frac{\sqrt{86}}{44}|E_{13}(POH(n))|+\frac{\sqrt{330}}{88}|E_{14}(POH(n))|+\\&&\frac{\sqrt{94}}{48}|E_{15}(POH(n))|
\end{eqnarray*}
	$\Longrightarrow ABC_4$ = $\frac{1}{440}(264\sqrt{29}+88\sqrt{930}+24\sqrt{2255}+5280(-1+n)+
	60\sqrt{330}(-1+n)n+24\sqrt{3410}(-1+n)+165\sqrt{94}(-1+n)^{2}+528\sqrt{35}n+132\sqrt{38}n+330\sqrt{46}(-1+n)n+60\sqrt{86}(-3+2n)+220\sqrt{35}()2-5n+3n^{2})$\\
	and from (8) we get\\
	$$GA_{5}(G)=\sum_{uv \in E(G)}\frac{2\sqrt{S_{u}S_{v}}}{(S_{u}+S_{v})}=\sum_{j=4}^{15}\sum_{uv \in E_j(G)}\frac{2\sqrt{S_{u}S_{v}}}{(S_{u}+S_{v})}.$$
	$GA_{5}(G)=2\bigg[|E_4(POH(n))|+\frac{2\sqrt{30}}{14}|E_5(POH(n))|+\frac{2\sqrt{2}}{3}|E_6(POH(n))|+\\\frac{\sqrt{55}}{8}|E_7(POH(n))|+
	|E_8(POH(n))|+\frac{\sqrt{15}}{4}|E_9(POH(n))|+\frac{2\sqrt{66}}{7}|E_{10}(POH(n))|+\frac{2\sqrt{2}}{3}|E_{11}(POH(n))|+
	\frac{2\sqrt{110}}{21}|E_{12}(POH(n))|+|E_{13}(POH(n))|+\frac{4\sqrt{33}}{23}|E_{14}(POH(n))|+|E_15(POH(n))|\bigg]$\\
	$\Longrightarrow GA_5$ = $8\sqrt{2}+6\sqrt{15}+\frac{8}{7}\sqrt{110}+\frac{48}{23}\sqrt{33}(-1+n)+\frac{3}{2}\sqrt{55}(-1+n)+\\ \frac{96}{17}\sqrt{66}(-1+n)-36n+\frac{48}{11}\sqrt{30}n+36n^{2}+8\sqrt{2}(2-5n+3n^{2})$.
\end{proof}

\subsection{Results for Triangular Prism network $TP(n)$}
In this section, we calculate certain degree based topological indices of Triangular Prism  network $TP(n)$ of dimension $n$. In the coming theorems we compute general Randi\'{c} index $R_{\alpha}(G)$ with $\alpha=\{1,-1,\frac{1}{2},-\frac{1}{2}\}$, $ABC$, $GA$, $ABC_{4}$ and $GA_{5}$ of $TP(n)$.
\begin{theorem}
	Consider the Triangular Prism TP(n) network $G\cong TP(n)$ for $n\in\mathbb{N}$. Then its general Randi\'{c} index is equal to
	\begin{equation*}
	R_{\alpha}(TP(n))=\left\{
	\begin{array}{ll}
	54n(-5+21n), & \hbox{$\alpha=1$;} \\
	18n(-3+\sqrt{2}+3(3+\sqrt{2})n), & \hbox{$\alpha=\frac{1}{2}$;} \\
	\frac{1}{6}n(4+21n), & \hbox{$\alpha=-1$;}\\
	n(\sqrt{2}+3(3+\sqrt{2})n), & \hbox{$\alpha=-\frac{1}{2}$.}
	\end{array}
	\right.
	\end{equation*}
\end{theorem}
\begin{proof}
	Let $G$ be the Triangular Prism network. The Triangular Prism network $TP(n)$ has $27n^2+3n$ vertices and the set of edges $TP(n)$  divided into three partitions based on the degree of end vertices. Table 3 shows such an edge partition of $TP(n)$. Thus from $(3)$ is follows that\\
	\begin{equation*}
	R_{\alpha}(G)=\sum\limits_{uv \in E(G)}(deg(u)deg(v))^{\alpha}.
	\end{equation*}
	Now we apply the formula,
	\begin{equation*}
	R_{1}(G)=\sum_{uv \in E_j(G)}\sum_{j=1}^3deg{(u)} \cdot deg{(v)}
	\end{equation*}
	By using edge partition given in Table 3, we get
	$$R_{1}(G)=9|E_1(TP(n))|+18|E_2(TP(n))|+36|E_3(TP(n))|$$
	$$\Longrightarrow R_{1}(G)= 54n(-5+21n).$$
	\textbf{For $\alpha=\frac{1}{2}$}\\
	We apply the formula,
	\begin{equation*}
	R_{\frac{1}{2}}(G)=\sum_{uv \in E_j(G)}\sum_{j=1}^3\sqrt{deg(u) \cdot deg(v)}.
	\end{equation*}
	By using edge partition given in Table 3, we get
	$$R_{\frac{1}{2}}(G)=3|E_1(TP(n))|+3\sqrt{2}|E_2(TP(n))|+6|E_3(TP(n))|$$
	$$\Longrightarrow R_{\frac{1}{2}}(G)= 18n(-3+\sqrt{2}+3(3+\sqrt{2})n).$$
	\\
	\textbf{For $\alpha=-1$}\\
	We apply the formula,
	\begin{equation*}
	R_{-1}(G)=\sum_{uv \in E_j(G)}\sum_{j=1}^3\frac{1}{deg(u) \cdot deg(v)}.
	\end{equation*}
	
	$$R_{-1}(G)=\frac{1}{9}|E_1(TP(n))|+\frac{1}{18}|E_2(TP(n))|+\frac{1}{36}|E_3(TP(n))|$$
	$$\Longrightarrow R_{-1}(G)=  \frac{1}{6}n(4+21n).$$
	\\
	\textbf{For $\alpha=-\frac{1}{2}$}\\
	We apply the formula,
	\begin{equation*}
	R_{-\frac{1}{2}}(G)=\sum_{uv \in E_j(G)}\sum_{j=1}^3\frac{1}{\sqrt{deg(u) \cdot deg(v)}}.
	\end{equation*}
	$$R_{-\frac{1}{2}}(G)=\frac{1}{23}|E_1(TP(n))|+\frac{1}{3\sqrt{2}}|E_2(TP(n))|+\frac{1}{6}|E_3(TP(n))|$$
	\begin{equation*}
	\Longrightarrow R_{-\frac{1}{2}}(G)=  n(\sqrt{2}+3(3+\sqrt{2})n).
	\end{equation*}
\end{proof}

In the following theorem, we compute first Zagreb index of Triangular Prism network $TP(n)$.
\begin{theorem}
	For Triangular Prism network $G\cong TP(n)$ for $n\in\mathbb{N}$. Then its first Zagreb index is equal to
	\begin{equation*}
	M_{1}(TP(n))=54n(-1+9n).
	\end{equation*}
\end{theorem}
\begin{proof}
	Let $G$ be the Triangular Prism network $TP(n)$. By using edge partition from Table 3, the result follows. From (4) we have
	\begin{equation*}
	M_{1}(TP(n))=\sum_{uv \in E(G)}(deg(u)+deg(v))=\sum_{uv \in E_j(G)}\sum_{j=1}^3(deg(u)+deg(v)).
	\end{equation*}
	$$M_{1}(TP(n))=6|E_1(TP(n))|+9|E_2(TP(n))|+12|E_3(TP(n))|$$
	By doing some calculation, we get\\
	$$\Longrightarrow M_{1}(TP(n))=54n(-1+9n).$$
\end{proof}

Now, we compute $ABC$, $GA$, $ABC_{4}$ and $GA_{5}$ indices of Triangular Prism network $TP(n)$.
\begin{theorem}
	Let $G\cong TP(n)$ be the Triangular Prism network, then\\ \\
	$ABC(G)$ = $n(4-2\sqrt{10}+\sqrt{14}+3(4+\sqrt{10}+\sqrt{14})n)$, for $n\geq1$ \\ \\
	$GA(G)$ = $2n(-3+2\sqrt{2}+6(3+\sqrt{2})n)$, for $n\geq1$\\ \\
	$ABC_4(G)$ = $4-\frac{\sqrt{2}}{3}-\frac{4\sqrt{13}}{3}+\sqrt{\frac{74}{5}}+\sqrt{17}-\frac{5\sqrt{22}}{3}-\frac{4\sqrt{37}}{3}+\frac{3\sqrt{58}}{5}+\\  \frac{2}{45}(-225+60\sqrt{2}+20\sqrt{13}+15\sqrt{22}+30\sqrt{37}+15\sqrt{57}-\\ 27\sqrt{58}+3\sqrt{330})n+[6+3\sqrt{\frac{11}{2}}+\frac{3\sqrt{58}}{5}]n^{2}$, for $n\geq2$\\ \\
	$GA_5(G)$ = $-\frac{444}{13}+\frac{280\sqrt{2}}{17}-\frac{36\sqrt{5}}{7}+\frac{5760\sqrt{10}}{1729}+[-\frac{24}{13}+\frac{48\sqrt{3}}{7}+\frac{36\sqrt{5}}{7}-\\ \frac{636\sqrt{10}}{133}+\frac{3\sqrt{15}}{2}]n+\frac{36}{7}(7+\sqrt{10})n^{2}$, for $n\geq2$
\end{theorem}
\begin{proof}
	By using edge partition given in Table 3 , we get the result. From $(5)$ it follows that
	$$ABC(G)=\sum_{uv \in E(G)}\sqrt{\frac{deg(u)+deg(v)-2}{deg(u)\cdot deg(v)}}=\sum_{uv \in E_j(G)}\sum_{j=1}^3\sqrt{\frac{deg(u)+deg(v)-2}{deg(u)\cdot deg(v)}}.$$
	$$ABC(G)=\frac{2}{3}|E_1(TP(n))|+\frac{\sqrt{14}}{6}|E_2(TP(n))|+\frac{\sqrt{10}}{6}|E_3(TP(n))|$$
	By doing some calculation, we get\\
	$$\Longrightarrow ABC(G)=n(4-2\sqrt{10}+\sqrt{14}+3(4+\sqrt{10}+\sqrt{14})n).$$
	from (6) we get
	$$GA(G)= \sum_{uv \in E(G)} \frac{2\sqrt{deg(u)deg(v)}}{(deg(u)+deg(v))}=\sum_{uv \in E_j(G)}\sum_{j=1}^3\frac{2\sqrt{deg(u)deg(v)}}{(deg(u)+deg(v))}.$$
	By doing some calculation, we get\\
	$$GA(G)=|E_1(TP(n))|+\frac{2\sqrt{2}}{3}|E_2(TP(n))|+|E_3(TP(n))|$$
	$$\Longrightarrow GA =2n(-3+2\sqrt{2}+6(3+\sqrt{2})n).$$
	Let us consider an edge partition based on degree sum of neighbors of end vertices. The edge set $E(TP(n))$ divided into twelve edge partitions $E_j(TP(n)), \\4\leq j \leq15$, where the edge partition $E_4(TP(n))$ contains $12n$ edges $uv$ with $S_u=9 and S_v=12$, the edge partition $E_5(TP(n))$ contains $6n$ edges $uv$ with $S_u=9$ and $S_v=15$, the edge partition $E_6(TP(n))$ contains $18n^{2}-12$ edges $uv$ with $S_u=S_v=12$, the edge partition $E_7(TP(n))$ contains $12$ edges $uv$ with $S_u=12$ and $S_v=24$, the edge partition $E_8(TP(n))$ contains $24n-24$ edges $uv$ with $S_u=12$ and $S_v=27$, the edge partition $E_9(TP(n))$ contains $18n^2-30n+12$ edges $uv$ with $S_u=12$ and $S_v=30$, the edge partition $E_{10}(TP(n))$ contains $12$ edges $uv$ with $S_u=15$ and $S_v=24$, the edge partition $E_{11}(TP(n))$ contains $12n-12$ edges $uv$ with $S_u=15$ and $S_v=27$, the edge partition $E_{12}(TP(n))$ contains $12$ edges $uv$ with $S_u=24$ and $S_v=27$, the edge partition $E_{13}(TP(n))$ contains $12n-18$ edges $uv$ with $S_u=S_v=27$, the edge partition $E_{14}(TP(n))$ contains $12n-12$ edges $uv$ with $S_u=27$ and $S_v=30$ and the edge partition $E_{15}(TP(n))$ contains $18n^2-36n+18$ edges $uv$ with $S_u=S_v=30$.\\ From (7) we get
	$$ABC_{4}(G)=\sum_{uv \in E(G)} \sqrt{\frac{S_{u}+S_{v}-2}{S_{u}S_{v}}}=\sum_{j=4}^{15}\sum_{uv \in E_j(G)}\sqrt{\frac{S_{u}+S_{v}-2}{S_{u}S_{v}}}.$$
	$ABC_{4}(G)=\frac{\sqrt{57}}{18}|E_4(TP(n))|+\frac{\sqrt{330}}{45}|E_5(TP(n))|+\frac{\sqrt{22}}{12}|E_6(TP(n))|+\\\frac{\sqrt{17}}{12}|E_7(TP(n))|+
	\frac{\sqrt{37}}{18}|E_8(TP(n))|+\frac{1}{3}|E_9(TP(n))|+\frac{\sqrt{370}}{60}|E_{10}(TP(n))|+\frac{2\sqrt{2}}{9}|E_{11}(TP(n))|+
	\frac{7\sqrt{2}}{36}|E_{12}(TP(n))|+\frac{2\sqrt{13}}{27}|E_{13}(TP(n))|+\\\frac{\sqrt{22}}{18}|E_{14}(TP(n))|+
	\frac{\sqrt{58}}{30}|E_{15}(TP(n))|$\\
	$\Longrightarrow ABC_4$ = $4-\frac{\sqrt{2}}{3}-\frac{4\sqrt{13}}{3}+\sqrt{\frac{74}{5}}+\sqrt{17}-\frac{5\sqrt{22}}{3}-\frac{4\sqrt{37}}{3}+\frac{3\sqrt{58}}{5}+\\ \frac{2}{45}(-225+60\sqrt{2}+20\sqrt{13}+15\sqrt{22}+30\sqrt{37}+15\sqrt{57}-\\ 27\sqrt{58}+3\sqrt{330})n+[6+3\sqrt{\frac{11}{2}}+\frac{3\sqrt{58}}{5}]n^{2}$\\
	and from (8) we get\\
	$$GA_{5}(G)=\sum_{uv \in E(G)}\frac{2\sqrt{S_{u}S_{v}}}{(S_{u}+S_{v})}=\sum_{j=4}^{15}\sum_{uv \in E_j(G)}\frac{2\sqrt{S_{u}S_{v}}}{(S_{u}+S_{v})}.$$\\
	$GA_{5}(G)=\frac{4\sqrt{3}}{7}|E_4(TP(n))|+\frac{\sqrt{15}}{4}|E_5(TP(n))|+|E_6(TP(n))|+\\\frac{2\sqrt{2}}{3}|E_7(TP(n))|+
	\frac{12}{13}|E_8(TP(n))|+\frac{2\sqrt{10}}{7}|E_9(TP(n))|+\\\frac{4\sqrt{10}}{13}|E_{10}(TP(n))|+\frac{3\sqrt{5}}{7}|E_{11}(TP(n))|
	+\frac{12\sqrt{2}}{17}|E_{12}(TP(n))|+\\|E_{13}(TP(n))|+\frac{6\sqrt{10}}{13}|E_{19}(TP(n))|+|E_{15}(TP(n))|$\\
	$\Longrightarrow GA_5$ = $-\frac{444}{13}+\frac{280\sqrt{2}}{17}-\frac{36\sqrt{5}}{7}+\frac{5760\sqrt{10}}{1729}+[-\frac{24}{13}+\frac{48\sqrt{3}}{7}+\frac{36\sqrt{5}}{7}-\\ \frac{636\sqrt{10}}{133}+\frac{3\sqrt{15}}{2}]n+\frac{36}{7}(7+\sqrt{10})n^{2}$.
\end{proof}


Now, we calculate certain degree based topological indices of Dominating Planar Octahedral  network $DPOH(n)$. In the coming theorems we compute general Randi\'{c} index $R_{\alpha}(G)$ with $\alpha=\{1,-1,\frac{1}{2},-\frac{1}{2}\}$, $ABC$, $GA$, $ABC_{4}$ and $GA_{5}$ of $DPOH(n)$.
\begin{theorem}
	Consider the Dominating Planar Octahedral network $G\cong DPOH(n)$ for $n\in\mathbb{N}$. Then its general Randi\'{c} index is equal to
	\begin{equation*}
	R_{\alpha}(RTSL(n))=\left\{
	\begin{array}{ll}
	288(11-31n+27n^{2}), & \hbox{$\alpha=1$;} \\
	24(11+6\sqrt{2}-(31+18\sqrt{2})n+9(3+2\sqrt{2})n^{2}), & \hbox{$\alpha=\frac{1}{2}$;} \\
	\frac{9}{32}(7-23n+27n^{2}), & \hbox{$\alpha=-1$;}\\
	\frac{3}{4}(7+6\sqrt{2}-(23+18\sqrt{2})n+9(3+2\sqrt{2})n^{2}), & \hbox{$\alpha=-\frac{1}{2}$.}
	\end{array}
	\right.
	\end{equation*}
\end{theorem}
\begin{proof}
	Let $G$ be the Dominating Planar Octahedral network. The Dominating Planar Octahedral network $DPOH(n)$ has $81^{2}-75n+24$ vertices. And the edges of $DPOH(n)$ has $216^{2}-216+72$. The edge set of $DPOH(n)$ is divided into three partitions based on the degree of end vertices. The first edge partition $E_1(DPOH(n))$ contains $54n^{2}-30n+6$ edges $uv$, where $deg(u)=deg(v)=4$. The second edge partition $E_2(DPOH(n))$ contains $108^{2}-108n+36$ edges $uv$, where $deg(u)=4$ and $deg(v)=8$. The third edge partition $E_3(DPOH(n))$ contains $54n^2-78n+30$ edges $uv$, where $deg(u)=deg(v)=8$. Table 5 shows such an edge partition of $DPOH(n)$. Thus from $(3)$ is follows that\\
	\begin{equation*}
	R_{\alpha}(G)=\sum\limits_{uv \in E(G)}(deg(u)deg(v))^{\alpha}.
	\end{equation*}
	Now we apply the formula,
	\begin{equation*}
	R_{1}(G)=\sum_{uv \in E_j(G)}\sum_{j=1}^3deg{(u)} \cdot deg{(v)}.
	\end{equation*}
	By using edge partition given in Table 5, we get
	$$R_{1}(G)=16|E_1(DPOH(n))|+32|E_2(DPOH(n))|+64|E_3(DPOH(n))|$$
	$$\Longrightarrow R_{1}(G)= 288(11-31n+27n^{2}).$$
	\textbf{For $\alpha=\frac{1}{2}$}\\
	We apply the formula,
	\begin{equation*}
	R_{\frac{1}{2}}(G)=\sum_{uv \in E_j(G)}\sum_{j=1}^3\sqrt{deg(u) \cdot deg(v)}.
	\end{equation*}
	By using edge partition given in Table 5, we get
	$$R_{\frac{1}{2}}(G)=4|E_1(DPOH(n))|+4\sqrt{2}|E_2(DPOH(n))|+8|E_3(DPOH(n))|$$
	$$\Longrightarrow R_{\frac{1}{2}}(G)=  24(11+6\sqrt{2}-(31+18\sqrt{2})n+9(3+2\sqrt{2})n^{2}).$$
	\\
	\textbf{For $\alpha=-1$}\\
	We apply the formula,
	\begin{equation*}
	R_{-1}(G)=\\sum_{uv \in E_j(G)}\sum_{j=1}^3\frac{1}{deg(u) \cdot deg(v)}.
	\end{equation*}
	
	$$R_{-1}(G)=\frac{1}{16}|E_1(DPOH(n))|+\frac{1}{32}|E_2(DPOH(n))|+\frac{1}{64}|E_3(DPOH(n))|$$
	$$\Longrightarrow R_{-1}(G)= \frac{9}{32}(7-23n+27n^{2}).$$
	\\
	\textbf{For $\alpha=-\frac{1}{2}$}\\
	We apply the formula,
	\begin{equation*}
	R_{-\frac{1}{2}}(G)=\sum_{uv \in E_j(G)}\sum_{j=1}^3\frac{1}{\sqrt{deg(u) \cdot deg(v)}}.
	\end{equation*}
	$$R_{-\frac{1}{2}}(G)=\frac{1}{4}|E_1(DPOH(n))|+\frac{1}{4\sqrt{2}}|E_2(DPOH(n))|+\frac{1}{8}|E_3(DPOH(n))|$$
	\begin{equation*}
	\Longrightarrow R_{-\frac{1}{2}}(G)=\frac{3}{4}(7+6\sqrt{2}-(23+18\sqrt{2})n+9(3+2\sqrt{2})n^{2}).
	\end{equation*}
\end{proof}

In the following theorem, we compute first Zagreb index of Dominating Planar Octahedral  network $DPOH(n)$.
\begin{theorem}
	For Dominating Planar Octahedral network $DPOH(n)$, the first Zagreb index is equal to
	\begin{equation*}
	M_{1}(DPOH(n))=96(10-29+27n^{2}).
	\end{equation*}
\end{theorem}
\begin{proof}
	Let $G$ be the Dominating Planar Octahedral network $DPOH(n)$. By using edge partition from Table 5, the result follows. From (4) we have
	\begin{equation*}
	M_{1}(DPOH(n))=\sum_{uv \in E(G)}(deg(u)+deg(v))=\sum_{uv \in E_j(G)}\sum_{j=1}^3(deg(u)+deg(v)).
	\end{equation*}
	$$M_{1}(DPOH(n))=8|E_1(DPOH(n))|+12|E_2(DPOH(n))|+16|E_3(DPOH(n))|$$
	By doing some calculation, we get\\
	$$\Longrightarrow M_{1}(DPOH(n))=96(10-29+27n^{2}).$$
\end{proof}
\qed
Now, we compute $ABC$, $GA$, $ABC_{4}$ and $GA_{5}$ indices of Dominating Planar Octahedral  network $DPOH(n)$.
\begin{theorem}
	Let $G\cong DPOH(n)$ be the  network, then\\ \\
	$ABC(G)$ = $\frac{1}{8}(72\sqrt{5}(1-3n+3n^{2})+6\sqrt{14}(5-13n+9n^{2})+12\sqrt{6}(1-5n+9n^{2}))$, for $n\geq1$ \\ \\
	$GA(G)$ = $12(3+2\sqrt{2})(1-3n+3n^{2})$, for $n\geq1$\\ \\
	$ABC_4(G)$ = $\frac{1}{440}(66\sqrt{78}+5280(-1+n)+120\sqrt{330}(-1+n)+\\ 24\sqrt{3410}(-1+n)+264\sqrt{29}n+88\sqrt{930}n+\\ 528\sqrt{35}(-1+2n)+132\sqrt{38}(-1+2n)+\\ 330\sqrt{46}(2-5n+3n^{2})+55\sqrt{94}(10-19n+9n^{2})+\\ 220\sqrt{35}(8-17n+9n^{2}))$, for $n\geq3$\\ \\
	$GA_5(G)$ =$96+\frac{96}{23}\sqrt{33}(-1+n)+\frac{3}{2}\sqrt{55}(-1+n)+\frac{96}{17}\sqrt{66}(-1+n)+\\ \frac{8}{7}\sqrt{110}(-1+n)-192n+8\sqrt{2}n+108n^{2}+\ \frac{48}{11}\sqrt{30}(-1+2n)+8\sqrt{2}(8-17n+9n^{2})$, for $n\geq3$
\end{theorem}
\begin{proof}
	By using edge partition given in Table 5 , we get the result. From $(5)$ it follows that
	$$ABC(G)=\sum_{uv \in E(G)}\sqrt{\frac{deg(u)+deg(v)-2}{deg(u)\cdot deg(v)}}=\sum_{uv \in E_j(G)}\sum_{j=1}^3\sqrt{\frac{deg(u)+deg(v)-2}{deg(u)\cdot deg(v)}}.$$
	$$ABC(G)=\frac{\sqrt{6}}{4}|E_1(DPOH(n))|+\frac{\sqrt{5}}{4}|E_2(DPOH(n))|+\frac{\sqrt{14}}{8}|E_3(RTSL(n))|$$
	By doing some calculation, we get\\
	$$\Longrightarrow ABC(G)=\frac{1}{8}(72\sqrt{5}(1-3n+3n^{2})+6\sqrt{14}(5-13n+9n^{2})+12\sqrt{6}(1-5n+9n^{2}))$$
	from (6) we get
	$$GA(G)= \sum_{uv \in E(G)} \frac{2\sqrt{deg(u)deg(v)}}{(deg(u)+deg(v))}=\sum_{uv \in E_j(G)}\sum_{j=1}^3\frac{2\sqrt{deg(u)deg(v)}}{(deg(u)+deg(v))}.$$
	By doing some calculation, we get\\
	$$GA(G)=|E_1(DPOH(n))|+\frac{2\sqrt{2}}{3}|E_2(DPOH(n))|+|E_3(DPOH(n))|$$
	$$\Longrightarrow GA =12(3+2\sqrt{2})(1-3n+3n^{2}).$$
	Let us consider an edge partition based on degree sum of neighbors of end vertices. Then the edge set $E(DPOH(n))$ can be divided into twelve edge partitions $E_j(DPOH(n)),4\leq j \leq15$, where the edge partition $E_4(DPOH(n))$ contains $12n-6$ edges $uv$ with $S_u=S_v=20$, the edge partition $E_5(DPOH(n))$ contains $48n-24$ edges $uv$ with $S_u=20$ and $S_v=24$, the edge partition $E_6(DPOH(n))$ contains $12n$ edges $uv$ with $S_u=20$ and $S_v=40$, the edge partition $E_7(DPOH(n))$ contains $12n-12$ edges $uv$ with $S_u=20$ and $S_v=44$, the edge partition $E_8(DPOH(n))$ contains $54n^{2}-90n+36$ edges $uv$ with $S_u=24$ and $S_v=24$, the edge partition $E_9(DPOH(n))$ contains $24n$ edges $uv$ with $S_u=24$ and $S_v=40$, the edge partition $E_{10}(DPOH(n))$ contains $48n-48$ edges $uv$ with $S_u=24$ and $S_v=44$, the edge partition $E_{11}(DPOH(n))$ contains $108n^{2}-204n+96$ edges $uv$ with $S_u=24$ and $S_v=48$, the edge partition $E_{12}(DPOH(n))$ contains $6$ edges $uv$ with $S_u=40$ and $S_v=40$, the edge partition $E_{13}(DPOH(n))$ contains $12n-12$ edges $uv$ with $S_u=40 and S_v=44$  and the edge partition $E_{14}(DPOH(n))$ contains $24n-24$ edges $uv$ with $S_u=44$ and $S_v=48$ and the edge partition $E_{15}(DPOH(n)$ contains $54n^{2}-114n+60$ edges $uv$ with $S_u=S_v=48.\\ From  (7)  we   get , \\
	$$ABC_{4}(G)=\sum_{uv \in E(G)} \sqrt{\frac{S_{u}+S_{v}-2}{S_{u}S_{v}}}=\sum_{j=4}^{16}\sum_{uv \in E_j(G)}\sqrt{\frac{S_{u}+S_{v}-2}{S_{u}S_{v}}}$$\\
	$$ABC_{4}(G)$=$\frac{\sqrt{38}}{20}|E_4(DPOH(n))|+\frac{\sqrt{35}}{20}|E_5(DPOH(n))|+\frac{\sqrt{29}}{20}|E_6(DPOH(n))|+\\0.2654|E_7(DPOH(n))|+
	\frac{\sqrt{46}}{4}|E_8(DPOH(n))|+0.2541|E_9(DPOH(n))|+\frac{1}{4}|E_{10}(DPOH(n))|+\frac{\sqrt{35}}{24}|E_{11}(DPOH(n))|+
	0.2518|E_{12}(DPOH(n))|+\frac{\sqrt{86}}{44}|E_{13}(DPOH(n))|+\frac{\sqrt{330}}{88}|E_{14}(DPOH(n))|+\frac{\sqrt{94}}{48}|E_{15}(DPOH(n))|$\\
	$\Longrightarrow ABC_4$ = $ \frac{1}{440}(66\sqrt{78}+5280(-1+n)+120\sqrt{330}(-1+n)+\\ 24\sqrt{3410}(-1+n)+264\sqrt{29}n+88\sqrt{930}n+\\ 528\sqrt{35}(-1+2n)+132\sqrt{38}(-1+2n)+\\ 330\sqrt{46}(2-5n+3n^{2})+55\sqrt{94}(10-19n+9n^{2})+\\ 220\sqrt{35}(8-17n+9n^{2}))$\\
	and from (8) we get\\
	$$GA_{5}(G)=\sum_{uv \in E(G)}\frac{2\sqrt{S_{u}S_{v}}}{(S_{u}+S_{v})}=\sum_{j=4}^{15}\sum_{uv \in E_j(G)}\frac{2\sqrt{S_{u}S_{v}}}{(S_{u}+S_{v})}$$\\
	$GA_{5}(G)=|E_4(DPOH(n))|+\frac{2\sqrt{30}}{14}|E_5D(POH(n))|+\frac{2\sqrt{2}}{3}|E_6(DPOH(n))|+\\\frac{\sqrt{55}}{8}|E_7(DPOH(n))|+
	|E_8(DPOH(n))|+\frac{\sqrt{15}}{4}|E_9(DPOH(n))|+\frac{2\sqrt{66}}{7}|E_{10}(DPOH(n))|+\frac{2\sqrt{2}}{3}|E_{11}(DPOH(n))|+
	\frac{2\sqrt{110}}{21}|E_{12}(DPOH(n))|+|E_{13}(DPOH(n))|+\frac{4\sqrt{33}}{23}|E_{14}(DPOH(n))|+|E_15(DPOH(n))|$\\
	$\Longrightarrow GA_5$ = $96+\frac{96}{23}\sqrt{33}(-1+n)+\frac{3}{2}\sqrt{55}(-1+n)+\frac{96}{17}\sqrt{66}(-1+n)+\\ \frac{8}{7}\sqrt{110}(-1+n)-192n+8\sqrt{2}n+108n^{2}+\\ \frac{48}{11}\sqrt{30}(-1+2n)+8\sqrt{2}(8-17n+9n^{2})$.
\end{proof}

The Comparison graphs of $ABC$, $GA$, $ABC_{4}$ and $GA_{5}$ for Planar Octahedron network $POH(n)$ of dimension $n$, Triangular Prism network $TP(n)$ of dimension $n$ and Dominating Planar Octahedral network $DPOH(n)$ of dimension $n$ are shown in fig. $2$, fig. $4$ and fig. $6$ respectively.

\begin{table}
	\centering
	\begin{tabular}{|c|c|c|}
		\hline
		$(d_{u},d_{v})$ where $uv\in E(G)$ & $\text{Number of edges}$ \\\hline
		$(4,4)$ & $18n^{2}+12n$ \\\hline
		$(4,8)$ & $36n^2$ \\\hline
		$(8,8)$ & $18n^2-12n$\\
		\hline
	\end{tabular}
	\vspace{.2cm}
	\caption{Edge partition of Planar Octahedron network $POH(n)$ based on degrees of end vertices of each edge.}
\end{table}
\begin{table}
	\centering
	\begin{tabular}{|c|c|c|c|c|c|c|}
		\hline
		$(S_{u},S_{v})$ where $uv\in E(G)$ & $\text{Number of edges}$ \\\hline
		$(20,20)$ & $6n$ \\\hline
		$(20,24)$ & $24n$ \\\hline
		$(20,40)$ & $12$ \\\hline
		$(20,44)$ & $12n-12$ \\\hline
		$(24,24)$ & $18n^{2}-18n$ \\\hline
		$(24,40)$ & $24$ \\\hline
		$(24,44)$ & $48n-48$ \\\hline
		$(24,48)$ & $36n^{^{2}}-60n+24$ \\\hline
		$(40,44)$ & $12$ \\\hline
		$(44,44)$ & $12n-18$ \\\hline
		$(44,48)$ & $12n-12$ \\
		$(48,48)$ & $18n^{2}-36n+18$\\
		\hline
	\end{tabular}
	\vspace{.2cm}
	\caption{Edge partition of David Derived network $DD(n)$ based on degrees of end vertices of each edge.}
\end{table}
\begin{figure}
	\centering
	\caption{Comparison of $ABC$, $GA$, $ABC_{4}$ and $GA_{5}$ for $DD(n)$}\label{fig 1}
\end{figure}

\begin{table}
	\centering
	\begin{tabular}{|c|c|}
		\hline
		$(d_{u},d_{v})$ where $uv\in E(G)$ & $\text{Number of edges}$ \\\hline
		$(3,3)$ & $18n^{2}+6n$ \\\hline
		$(3,6)$ & $18n^2+6n$ \\\hline
		$(6,6)$ & $18n^2-12n$ \\
		\hline
	\end{tabular}
	\vspace{.2cm}
	\caption{Edge partition of Triangular Prism network $TP(n)$ based on degrees of end vertices of each edge.}
\end{table}
\begin{table}
	\centering
	\begin{tabular}{|c|c|}
		\hline
		$(S_{u},S_{v})$ where $uv\in E(G)$ & $\text{Number of edges}$ \\\hline
		$(9,12)$ & $12n$ \\\hline
		$(9,15)$ & $6n$ \\\hline
		$(12,12)$ & $8n^{2}-12$ \\\hline
		$(12,24)$ & $12$ \\\hline
		$(12,27)$ & $24n-24$ \\\hline
		$(12,30)$ & $18n^2-30n+12$ \\\hline
		$(15,24)$ & $12$ \\\hline
		$(15,27)$ & $12n-12$ \\\hline
		$(24,27)$ & $12$ \\\hline
		$(27,27)$ & $12n-18$ \\\hline
		$(27,30)$ & $12n-12$ \\\hline
		$(30,30)$ & $18n^2-36n+18$ \\
		\hline
	\end{tabular}
	\vspace{.2cm}
	\caption{Edge partition of Triangular Prism network $TP(n)$ based on degrees of end vertices of each edge.}\label{abc}
\end{table}

\begin{table}
	\centering
	\begin{tabular}{|c|c|}
		\hline
		$(d_{u},d_{v})$ where $uv\in E(G)$ & $\text{Number of edges}$ \\\hline
		$(4,4)$ & $54n^{2}-30n+6$ \\\hline
		$(4,8)$ & $108n^2-108n+36$ \\\hline
		$(8,8)$ & $54n^2-78n+30$ \\
		\hline
	\end{tabular}
	\vspace{.2cm}
	\caption{Edge partition of Dominating Planar Octahedral network $DPOH(n)$ based on degrees of end vertices of each edge.}
\end{table}
\begin{table}
	\centering
	\begin{tabular}{|c|c|}
		\hline
		$(S_{u},S_{v})$ where $uv\in E(G)$ & $\text{Number of edges}$ \\\hline
		$(20,20)$ & $12n-6$ \\\hline
		$(20,24)$ & $48n-24$ \\\hline
		$(20,40)$ & $12n$ \\\hline
		$(20,44)$ & $12n-12$ \\\hline
		$(24,24)$ & $54n^{2}-90+36n$ \\\hline
		$(24,40)$ & $24n$ \\\hline
		$(24,44)$ & $48n-48$ \\\hline
		$(24,48)$ & $108n^{^{2}}-204n+96$ \\\hline
		$(40,40)$ & $6$ \\\hline
		$(40,44)$ & $12n-12$ \\\hline
		$(44,48)$ & $24n-24$ \\
		$(48,48)$ & $54n^{2}-114n+60$\\
		\hline
	\end{tabular}
	\vspace{.2cm}
	\caption{Edge partition of Dominating Planar Octahedral network $DPOH(n)$ based on degrees of end vertices of each edge.}\label{abc}
\end{table}
The graphical representations of topological indices of these networks are depicted in Figure 5 and Figure 6 for certain values of $m$. By varying the different values of $m$, the graphs are increasing. These graphs shows the correctness of results.
\\
\noindent\textbf{Conclusion:}\\\\
In this paper, certain degree based topological indices, to be specific the Randi$\acute{c}$, Zagreb, $ABC$, $GA$, $ABC_4$ and $GA_5$ indices for planar octahedron networks just because and scientific shut equations for these systems were resolved which will help the general population working in system science to comprehend and investigate the basic topologies of these systems. Later on, we are intrigued to structure some new models/systems and after that review their topological records which will be very useful to comprehend their fundamental topologies.

\end{document}